\documentclass[11pt,a4paper]{amsart}
\usepackage{amsmath}
\usepackage[utf8]{inputenc}
\usepackage[english]{babel}
\usepackage[T1]{fontenc}
\usepackage{url}
\usepackage{mathrsfs}
\usepackage{tikz}
\usepackage{ amssymb }
\usepackage{amsthm}

\newtheorem{corollary*}{Corollary}
\theoremstyle{definition}
\newtheorem{definition}{Definition}
\theoremstyle{remark}
\newtheorem*{remark}{Remark}

\newcommand \M {\mathcal{M}}

\newcommand \ord{\operatorname{ord}}

\title{On irreducible meanders growth rate}
\author{Yury Belousov and Andrei Malyutin}
\thanks{The research is supported by the Foundation for the Advancement of Theoretical Physics and Mathematics ``BASIS''. The first author is partially supported by International Laboratory of Cluster Geometry NRU HSE, RF Government grant, ag. № 075-15-2021-608 from 08.06.2021}
\address{Yury Belousov\\
Faculty of Mathematics, National Research University HSE}
\email{bus99@yandex.ru}

\address{Andrei Malyutin\\
St.\,Petersburg Department of 
Steklov Institute of Mathematics\\
St.~Petersburg State University}
\email{malyutin@pdmi.ras.ru}

\begin{document}
\maketitle
\begin{abstract}
In this article, we provide upper and lower bounds for the growth rate of irreducible meanders.
The obtained upper bound implies that the proportion of irreducible meanders among all of the prime meanders of order $n$ approaches $0$ as $n$ approaches infinity.
\end{abstract}
\section{Introduction}
This paper is based on an analogy between decompositions of knots and recently discovered decomposition of meanders. In knot theory a well known theorem due to Schubert states that each knot admits a unique decomposition into prime knots. There are three types of prime knots: hyperbolic, torus, and satellite ones. There is a next level of decomposition (called JSJ decomposition) --- any satellite knot could be constructed from hyperbolic and torus ones. We can catch an analogy of this decomposition in the theory of meanders. The operation of connected sum of knots is analogous to the operation of concatenation of meanders (see Fig.~\ref{fig:concatenation}). But if an open meander cannot be constructed as a concatenation of two meanders it still can be decomposed into two types of \emph{simple} elements: we call them snakes and irreducible meanders. Irreducible meanders are analogous to hyperbolic knots, and snakes similar to torus knots. 

\begin{figure}[h]
    \centering
        \begin{tikzpicture}
        \node (plus) at (2.8, 0) {+};
            \draw (-0.5, 0) to (2.5, 0);
            \draw[thick] (-0.5,0.666667) to [out = 0, in = 90] (2, 0)
                to [out = -90, in = -90, distance = 12.5664] (1, 0)
                to [out = 90, in = 90, distance = 12.5664] (0, 0)
                to [out = -90, in = 180] (2.5, -0.666667);
        \end{tikzpicture}
        \begin{tikzpicture}
        \node (plus) at (2.8, 0) {=};
            \draw (-0.5, 0) to (2.5, 0);
            \draw[thick] (-0.5,0.666667) to [out = 0, in = 90] (2, 0)
                to [out = -90, in = -90, distance = 12.5664] (1, 0)
                to [out = 90, in = 90, distance = 12.5664] (0, 0)
                to [out = -90, in = 180] (2.5, -0.666667);
        \end{tikzpicture}
        \begin{tikzpicture}
            \draw (-0.5, 0) to (3, 0);
            \draw[thick] (-0.5,0.333333) to [out = 0, in = 90] (1, 0)
                to [out = -90, in = -90, distance = 6.28318] (0.5, 0)
                to [out = 90, in = 90, distance = 6.28318] (0, 0)
                to [out = -90, in = -90, distance = 31.4159] (2.5, 0)
                to [out = 90, in = 90, distance = 6.28318] (2, 0)
                to [out = -90, in = -90, distance = 6.28318] (1.5, 0)
                to [out = 90, in = 180] (3, 0.333333);
        \end{tikzpicture}

    \caption{Concatenation of two meanders.}
    \label{fig:concatenation}
\end{figure}
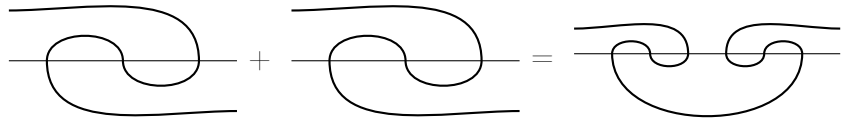

Torus knots are known to be rare (and so are snakes: there are precisely $1 + n\ \mathrm{mod}\ 2$ snakes of order $n$), And what can we say about hyperbolic knots? It was a conjecture of Adams~\cite{Ad94} stating that the proportion of hyperbolic knots among all of the prime knots of $n$ or fewer crossings converges to $1$ as $n$ approaches infinity. In a recent paper~\cite{BM19} we have disproved this conjecture. We conjecture, that this proportion converges to zero, but we are not able to prove it. We can ask analogous question about meanders, Does the proportion of irreducible meanders among all of the prime meanders of order $n$ converges to $0$ as $n$ approaches infinity? The answer to this question is the main result of the present paper. 

\subsection{Definitions}
\begin{definition}
An \emph{open meander} $(D, \{p_1, p_2,p_3,p_4\}, \{m, l\})$ is a triple of \begin{itemize}
    \item 2-dimensional disk $D$;
    \item four distinct points $p_1, p_2,p_3,p_4$ on $\partial D$ such that there exists a connected component of $\partial D \setminus \{p_1, p_2\}$ containing $\{p_3,p_4\}$;
    \item images $m$ and $l$ of smooth proper embeddings of the segment $[0;\,1]$ into $D$ such that $\partial m = \{p_1,p_3\}$, $\partial l = \{p_2,p_4\}$, and $m$ and $l$ intersect only transversally.
\end{itemize}
\end{definition}
\begin{definition}
We say that two meanders $$M=(D,\{p_1,p_2,p_3,p_4\}, \{m,l\})$$ and $$M'= (D',\{p_1',p_2',p_3',p_4'\},\{m',l'\})$$ are \emph{equivalent} if there exists a homeomorphism $f:D\to D'$ such that $f(m)=m'$, $f(l)=l'$, and $f(p_i)=p_i'$ for each $i=1,\dots,4$.  
\end{definition}
\begin{definition}
If $M=(D, m,l,p_1,p_2,p_3,p_4)$ is a meander, by the \emph{order of $M$} (denoted by $\ord(M)$) we mean the number of intersection points of $m$ and $l$. By $\M_n$ we denote the number of all equivalence classes of meanders of order $n$.
\end{definition}

In this paper we are going to present meanders via permutations, as intersection points has natural labeling by the integer numbers from $1$ to $n$ (see~\cite{C03} for details). We will also use some standard facts from the theory of meanders. The facts that we use can be found, e.\,g., in~\cite{C03}. 

Now we will give some new definitions. 

\begin{definition}[Irreducible meanders] \label{def:irreducible meander}
The meander with permutation $(a_1,\dots,a_n)$ is \emph{irreducible} if there is no $k_1, k_2 \in \mathbb{N}$ such that 
\begin{enumerate}
    \item $2 < k_2 - k_1 < n$,
    \item $\max\limits_{k_1\leq l \leq k_2} a_l\  -\   \min\limits_{k_1 \leq l \leq k_2} a_l = k_2 - k_1.$ 
\end{enumerate}
We denote the number of all equivalence classes irreducible meanders with $n$ double points by $\M^{(Irr)}_n$.
\end{definition}

\begin{definition}[Odd inserts]
Let $M$ be an open meander of order $n$ with permutation $(a_1, \dots, a_n)$, and let $M'$ be an open meander of odd order $m$ with permutation $(b_1,\dots,b_m)$. Let $M''$ be a meander with permutation $$(a_1',\dots,a_k', a_k + b_1, a_k + b_2, \dots, a_k + b_m,  a_{k+1}',\dots, a_n'),$$ where $$a_k' = 
\begin{cases}
a_k & a_k < k,\\
a_k + m & a_k > k.
\end{cases}$$ 
Then we say that $M''$ is obtained by the \emph{insert} of $M'$ into $M$ at the point $a_k$.
\end{definition}

\begin{definition}[Even inserts]
Let $M$ be an open meander of order $n$ with permutation $(a_1, \dots, a_n)$, let $M'$ be an open meander of even order $m$ with permutation $(b_1,\dots,b_m)$, and let $k$ be an integer such that $|a_k - a_{k+1}| = 1$. Let $M''$ be a meander with permutation $$(a_1',\dots,a_k', a_k + b_1', a_k + b_2', \dots, a_k + b_m',  a_{k+2}',\dots, a_n'),$$ where $$a_k' = 
\begin{cases}
a_k & a_k < k,\\
a_k + m & a_k > k
\end{cases}$$ and
$$b_k' = 
\begin{cases}
b_k & a_k < a_{k+1},\\
b_{m+1 - k} & a_k > a_{k+1}.
\end{cases}$$ 
Then we say that $M''$ is obtained by the \emph{insert} of $M'$ into $M$ at the point $a_k$.
\end{definition}

\begin{remark}
Definitions 1--3 have been inspired by geometric ideas, which will be explained in detail in the upcoming paper~\cite{B21}.
\end{remark}

\section{Growth rate of irreducible meanders}

\subsection{An upper bound on the growth rate of irreducible meanders}
Let $M$ be an irreducible meander of order $n$. We can choose $\frac{n}{k}$ distinct  double points and insert at each of this points a meander with permutation $(3,2,1)$ to obtain an open meander of order $n + \frac{2n}{k}$. Moreover, distinct subsets of double points produce non-equivalent meanders.
Thus we have the following inequalities:
\begin{align*}
    \M_{n+\frac{2n}{k}} &\geq  \binom{n}{\frac{n}{k}}\M^{(Irr)}_n,\\
    \lim_{n\to \infty} \sqrt[n]{\M_{n+\frac{2n}{k}}} &\geq  \limsup_{n\to \infty}  \sqrt[n]{\binom{n}{\frac{n}{k}}\M^{(Irr)}_n},\\
    \left(\lim_{n\to \infty} \sqrt[n]{\M_{n}}\right)^{\frac{k+2}{k}} &\geq   \frac{k}{(k-1)^{\frac{k-1}{k}}} \limsup_{n\to \infty} \sqrt[n]{\M^{(Irr)}_n},\\
    \limsup_{n\to \infty} \sqrt[n]{\M^{(Irr)}_n} &\leq \frac{(k-1)^{\frac{k-1}{k}}}{k} \left(\lim_{n\to \infty} \sqrt[n]{\M_{n}}\right)^{\frac{k+2}{k}}.
\end{align*}
It is proved in \cite{AP05} that $\lim\limits_{n\to \infty} \sqrt[n]{\Bar{\M}_{n}}\leq 12.901$, where $\Bar{\M}_{n}$ is the number of closed meanders with precisely $2n$ double points. It can be easily seen that $\M_{2n-1} = \Bar{\M}_{n}$ and $\Bar{\M}_{n} \leq \M_{2n} \leq n\Bar{\M}_{n}$ (see~\cite{C03} for details). From this we obtain the following estimate: $\lim\limits_{n\to \infty} \sqrt[n]{\M_n} \leq \sqrt{12.901}$. 
Now we get 
\begin{align}
\label{eq:lower bound}
     \limsup_{n\to \infty} \sqrt[n]{\M^{(Irr)}_n} \leq \frac{(k-1)^{\frac{k-1}{k}}}{k} 12.901^{\frac{k+2}{2k}}
\end{align}
for each $k>1$. The function on the left side of Eq.~$\eqref{eq:lower bound}$ reaches the minimum at $k\approx 13.901$, and finally we have the following estimate:
\begin{equation}
    \limsup_{n\to \infty} \sqrt[n]{\M^{(Irr)}_n} \leq 3.33341.
\end{equation}
\begin{remark}
This estimate could be improved if we insert more complicated meanders instead of meanders with permutation $(3,2,1)$.
\end{remark}
\begin{corollary*}
We have $\lim\limits_{n\to \infty} \frac{\M^{(Irr)}_n}{\M_n} = 0.$
\end{corollary*}
\begin{proof}
From the work~\cite{AP05} it follows that $\lim\limits_{n\to \infty} \sqrt[n]{\M_n} > 3.37343$. Now the corollary easily follows from the fact that $\limsup\limits_{n\to \infty} \sqrt[n]{\M^{(Irr)}_n}~<~\lim\limits_{n\to \infty} \sqrt[n]{\M_n}$.
\end{proof}
\begin{remark}
The growth rates of open meanders and of prime\footnote{A meander of order $n$ with permutation $(a_1,\dots,a_n)$ is \emph{prime} if there is no $k$ such that $1<k\leq n$, $\max\limits_{1\leq l < k} a_l\  -\   \min\limits_{1 \leq l <k} a_l = k - 1$ and $a_1\neq 1$.} 
open meanders are the same. Let $M$ be an open meander with permutation $(a_1,a_2,\dots,a_n)$, then if $n$ is even, a meander with permutation  $(a_1+2,a_2+2,\dots,a_n+2, n+3, 2,1)$ is indeed prime. If $n$ is odd then a meander with permutation $(a_1+2,a_2+2,\dots,a_n+2, 2,1)$ is prime. This implies that the proportion of irreducible meanders among all of the prime meanders of order $n$ converges to $0$ as $n$ approaches infinity. 
\end{remark}

\subsection{A lower bound on the growth rate of irreducible meanders}
Let us first describe the main idea of this subsection non-formally. We are going to construct an irreducible meander of order approximately $2n$ starting from a given meander $M$ of order $n$. 
So, let~$M$ be an arbitrary meander of order $n$, where $n$ is odd. We can construct irreducible meanders of orders $2n+32$ and $2n+35$ from $M$ by the following procedure. 
Consider a concatenation of a meander with the permutation $$(11, 10, 1, 2, 9, 12, 13, 8, 3, 4, 7, 14,15, 6, 5)$$ and $M$ (see example in Fig.~\ref{fig:example}(a), where $M$ is the meander with the permutation $(1,2,3,4,5)$). If we add another double point between the points $13$ and $14$ (see Fig.~\ref{fig:example}(b))\footnote{Note that after this operation $M$ is not a meander anymore!}, and then double the whole meander as in Fig.~\ref{fig:example}(c), we will obtain an irreducible meander of order $2n+32$. If after that we add three more double points in a suitable way (see example in Fig.~\ref{fig:example}(d)) then the resulting meander will be irreducible of order $2n+35$. If we start from a meander of even order then the procedure is a bit different but almost the same. 

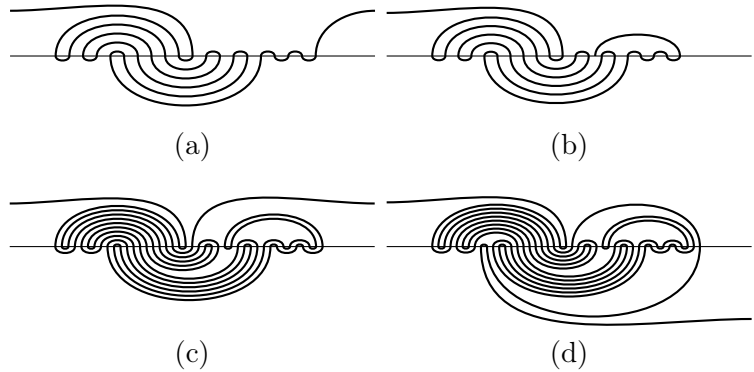
\begin{figure}[h]
    \centering
        \begin{tikzpicture}[scale = 1.2]
            \node (a) at (1.5, -1) {(a)};
            \draw (-0.5, 0) to (3.5, 0);
            \draw[thick] (-0.5,0.5) to [out = 0, in = 90] (1.5, 0)
            to [out = -90, in = -90, distance = 1.88496] (1.35, 0)
            to [out = 90, in = 90, distance = 16.9646] (0, 0)
            to [out = -90, in = -90, distance = 1.88496] (0.15, 0)
            to [out = 90, in = 90, distance = 13.1947] (1.2, 0)
            to [out = -90, in = -90, distance = 5.65487] (1.65, 0)
            to [out = 90, in = 90, distance = 1.88496] (1.8, 0)
            to [out = -90, in = -90, distance = 9.42478] (1.05, 0)
            to [out = 90, in = 90, distance = 9.42478] (0.3, 0)
            to [out = -90, in = -90, distance = 1.88496] (0.45, 0)
            to [out = 90, in = 90, distance = 5.65487] (0.9, 0)
            to [out = -90, in = -90, distance = 13.1947] (1.95, 0)
            to [out = 90, in = 90, distance = 1.88496] (2.1, 0)
            to [out = -90, in = -90, distance = 16.9646] (0.75, 0)
            to [out = 90, in = 90, distance = 1.88496] (0.6, 0)
            to [out = -90, in = -90, distance = 20.7345] (2.25, 0)
            to [out = 90, in = 90, distance = 1.88496] (2.4, 0)
            to [out = -90, in = -90, distance = 1.88496] (2.55, 0)
            to [out = 90, in = 90, distance = 1.88496] (2.7, 0)
            to [out = -90, in = -90, distance = 1.88496] (2.85, 0)
            to [out = 90, in = 180] (3.5, 0.5);
        \end{tikzpicture}
        \begin{tikzpicture}[scale = 1.2]
            \node (b) at (1.5, -1) {(b)};
            \draw (-0.5, 0) to (3.5, 0);
            \draw[thick] (-0.5,0.47619) to [out = 0, in = 90] (1.42857, 0)
            to [out = -90, in = -90, distance = 1.7952] (1.28571, 0)
            to [out = 90, in = 90, distance = 16.1568] (0, 0)
            to [out = -90, in = -90, distance = 1.7952] (0.142857, 0)
            to [out = 90, in = 90, distance = 12.5664] (1.14286, 0)
            to [out = -90, in = -90, distance = 5.38559] (1.57143, 0)
            to [out = 90, in = 90, distance = 1.7952] (1.71429, 0)
            to [out = -90, in = -90, distance = 8.97598] (1, 0)
            to [out = 90, in = 90, distance = 8.97598] (0.285714, 0)
            to [out = -90, in = -90, distance = 1.7952] (0.428571, 0)
            to [out = 90, in = 90, distance = 5.38559] (0.857143, 0)
            to [out = -90, in = -90, distance = 12.5664] (1.85714, 0)
            to [out = 90, in = 90, distance = 1.7952] (2, 0)
            to [out = -90, in = -90, distance = 16.1568] (0.714286, 0)
            to [out = 90, in = 90, distance = 1.7952] (0.571429, 0)
            to [out = -90, in = -90, distance = 19.7472] (2.14286, 0)
            to [out = 90, in = 90, distance = 1.7952] (2.28571, 0)
            to [out = -90, in = -90, distance = 1.7952] (2.42857, 0)
            to [out = 90, in = 90, distance = 1.7952] (2.57143, 0)
            to [out = -90, in = -90, distance = 1.7952] (2.71429, 0)
            to [out = 90, in = 90, distance = 8.97598] (1.785715, 0);
        \end{tikzpicture}
        \begin{tikzpicture}[scale = 1.2]
            \node (c) at (1.5, -1.2) {(c)};
            \draw (-0.5, 0) to (3.5, 0);
            \draw[thick] (-0.5,0.47619) to [out = 0, in = 90] (1.42857, 0)
            to [out = -90, in = -90, distance = 0.897598] (1.35714, 0)
            to [out = 90, in = 90, distance = 17.0544] (0, 0)
            to [out = -90, in = -90, distance = 2.69279] (0.214286, 0)
            to [out = 90, in = 90, distance = 11.6688] (1.14286, 0)
            to [out = -90, in = -90, distance = 6.28318] (1.64286, 0)
            to [out = 90, in = 90, distance = 0.897598] (1.71429, 0)
            to [out = -90, in = -90, distance = 8.07838] (1.07143, 0)
            to [out = 90, in = 90, distance = 9.87358] (0.285714, 0)
            to [out = -90, in = -90, distance = 2.69279] (0.5, 0)
            to [out = 90, in = 90, distance = 4.48799] (0.857143, 0)
            to [out = -90, in = -90, distance = 15.2592] (2.07143, 0)
            to [out = 90, in = 90, distance = 0.897598] (2.14286, 0)
            to [out = -90, in = -90, distance = 17.0544] (0.785714, 0)
            to [out = 90, in = 90, distance = 2.69279] (0.571429, 0)
            to [out = -90, in = -90, distance = 22.4399] (2.35714, 0)
            to [out = 90, in = 90, distance = 0.897598] (2.42857, 0)
            to [out = -90, in = -90, distance = 2.69279] (2.64286, 0)
            to [out = 90, in = 90, distance = 0.897598] (2.71429, 0)
            to [out = -90, in = -90, distance = 2.69279] (2.92857, 0)
            to [out = 90, in = 90, distance = 13.464] (1.85714, 0)
            to [out = -90, in = -90, distance = 0.897598] (1.92857, 0)
            to [out = 90, in = 90, distance = 11.6688] (2.85714, 0)
            to [out = -90, in = -90, distance = 0.897598] (2.78571, 0)
            to [out = 90, in = 90, distance = 2.69279] (2.57143, 0)
            to [out = -90, in = -90, distance = 0.897598] (2.5, 0)
            to [out = 90, in = 90, distance = 2.69279] (2.28571, 0)
            to [out = -90, in = -90, distance = 20.6448] (0.642857, 0)
            to [out = 90, in = 90, distance = 0.897598] (0.714286, 0)
            to [out = -90, in = -90, distance = 18.8496] (2.21429, 0)
            to [out = 90, in = 90, distance = 2.69279] (2, 0)
            to [out = -90, in = -90, distance = 13.464] (0.928571, 0)
            to [out = 90, in = 90, distance = 6.28318] (0.428571, 0)
            to [out = -90, in = -90, distance = 0.897598] (0.357143, 0)
            to [out = 90, in = 90, distance = 8.07838] (1, 0)
            to [out = -90, in = -90, distance = 9.87358] (1.78571, 0)
            to [out = 90, in = 90, distance = 2.69279] (1.57143, 0)
            to [out = -90, in = -90, distance = 4.48799] (1.21429, 0)
            to [out = 90, in = 90, distance = 13.464] (0.142857, 0)
            to [out = -90, in = -90, distance = 0.897598] (0.0714286, 0)
            to [out = 90, in = 90, distance = 15.2592] (1.28571, 0)
            to [out = -90, in = -90, distance = 2.69279] (1.5, 0)
            to [out = 90, in = 180] (3.5, 0.47619);
        \end{tikzpicture}
        \begin{tikzpicture}[scale = 1.2]
            \node (d) at (1.5, -1.2) {(d)};
            \draw (-0.5, 0) to (3.5, 0);
            \draw[thick] (-0.5,0.488889) to [out = 0, in = 90] (1.46667, 0)
            to [out = -90, in = -90, distance = 0.837758] (1.4, 0)
            to [out = 90, in = 90, distance = 17.5929] (0, 0)
            to [out = -90, in = -90, distance = 2.51327] (0.2, 0)
            to [out = 90, in = 90, distance = 12.5664] (1.2, 0)
            to [out = -90, in = -90, distance = 5.86431] (1.66667, 0)
            to [out = 90, in = 90, distance = 0.837758] (1.73333, 0)
            to [out = -90, in = -90, distance = 7.53982] (1.13333, 0)
            to [out = 90, in = 90, distance = 10.8909] (0.266667, 0)
            to [out = -90, in = -90, distance = 2.51327] (0.466667, 0)
            to [out = 90, in = 90, distance = 5.86431] (0.933333, 0)
            to [out = -90, in = -90, distance = 14.2419] (2.06667, 0)
            to [out = 90, in = 90, distance = 0.837758] (2.13333, 0)
            to [out = -90, in = -90, distance = 15.9174] (0.866667, 0)
            to [out = 90, in = 90, distance = 2.51327] (0.666667, 0)
            to [out = -90, in = -90, distance = 20.944] (2.33333, 0)
            to [out = 90, in = 90, distance = 0.837758] (2.4, 0)
            to [out = -90, in = -90, distance = 2.51327] (2.6, 0)
            to [out = 90, in = 90, distance = 0.837758] (2.66667, 0)
            to [out = -90, in = -90, distance = 2.51327] (2.86667, 0)
            to [out = 90, in = 90, distance = 12.5664] (1.86667, 0)
            to [out = -90, in = -90, distance = 0.837758] (1.93333, 0)
            to [out = 90, in = 90, distance = 10.8909] (2.8, 0)
            to [out = -90, in = -90, distance = 0.837758] (2.73333, 0)
            to [out = 90, in = 90, distance = 2.51327] (2.53333, 0)
            to [out = -90, in = -90, distance = 0.837758] (2.46667, 0)
            to [out = 90, in = 90, distance = 2.51327] (2.26667, 0)
            to [out = -90, in = -90, distance = 19.2684] (0.733333, 0)
            to [out = 90, in = 90, distance = 0.837758] (0.8, 0)
            to [out = -90, in = -90, distance = 17.5929] (2.2, 0)
            to [out = 90, in = 90, distance = 2.51327] (2, 0)
            to [out = -90, in = -90, distance = 12.5664] (1, 0)
            to [out = 90, in = 90, distance = 7.53982] (0.4, 0)
            to [out = -90, in = -90, distance = 0.837758] (0.333333, 0)
            to [out = 90, in = 90, distance = 9.21534] (1.06667, 0)
            to [out = -90, in = -90, distance = 9.21534] (1.8, 0)
            to [out = 90, in = 90, distance = 2.51327] (1.6, 0)
            to [out = -90, in = -90, distance = 4.18879] (1.26667, 0)
            to [out = 90, in = 90, distance = 14.2419] (0.133333, 0)
            to [out = -90, in = -90, distance = 0.837758] (0.0666667, 0)
            to [out = 90, in = 90, distance = 15.9174] (1.33333, 0)
            to [out = -90, in = -90, distance = 2.51327] (1.53333, 0)
            to [out = 90, in = 90, distance = 17.5929] (2.93333, 0)
            to [out = -90, in = -90, distance = 29.3215] (0.6, 0)
            to [out = 90, in = 90, distance = 0.837758] (0.533333, 0)
            to [out = -90, in = 180] (3.5, -0.8);
        \end{tikzpicture}
    \caption{Constructing an irreducible meander from a given one.}
    \label{fig:example}
\end{figure}

The simplest way to formalize this is to describe this procedure on the level of permutations.
Let $M$ be an arbitrary open meander of odd order $n$ with permutation  $(a_1,\dots,a_n)$ . Then the meanders with permutations 
\begin{align*}
    (&21,\ 20,\  1,\  4,\ 17,\ 24,\ 25,\ 16,\  5,\  8,\ 13,\ 30,\ 31,\ 12,\  9, \\
    &32+2a_1,\ 32+2a_2-1,\ 32+2a_3,\ \dots,\ 32+2a_n,\\
    &27,\ 28,\\
    &32+2a_n-1,\ 32+2a_{n-1},\ 32+2a_{n-2}-1,\ \dots,\ 32+2a_1-1,\\
    &10,\ 11,\ 32,\ 29,\ 14,\  7,\  6,\ 15,\ 26,\ 23,\ 18,\  3,\  2,\ 19,\ 22)
\end{align*}
and
\begin{align*}
    (&23,\ 22,\  1,\  4,\ 19,\ 26,\ 27,\ 18,\  5,\  8,\ 15,\ 32,\ 33,\ 14,\ 11,\\
    &34+2a_1,\ 34+2a_2-1,\ 34+2a_3,\ \dots,\ 34+2a_n,\\
    &29,\ 30,\\
    &34+2a_n-1,\ 34+2a_{n-1},\ 34+2a_{n-2}-1,\ \dots,\ 34+2a_1-1,\\
    &12,\ 13,\ 34,\ 31,\ 16,\  9,\  8,\ 17,\ 28,\ 25,\ 20,\  5,\  4,\ 21,\ 24,\\
    &2n+35,\ 10,\ 9)
\end{align*}
are irreducible meanders of orders $2n+32$ and $2n+35$, respectively. The proof of irreducibleity is straightforward by Definition~\ref{def:irreducible meander}. 

If we start from a meander with permutation $(a_1,\dots,a_n)$ of even order then the permutations 
\begin{align*}
    (&23,\ 22,\  1,\  4,\ 19,\ 26,\ 27,\ 18,\  5,\  8,\ 15,\ 30,\ 31,\ 14,\ 11,\\
    &32+2a_1,\ 32+2a_2-1,\ 32+2a_3,\ \dots,\ 32+2a_n,\\
    &10,\ 9,\\
    &32+2a_n-1,\ 32+2a_{n-1},\ 32+2a_{n-2}-1,\ \dots,\ 32+2a_1-1,\\
    &12,\ 13,\ 32,\ 29,\ 16,\  7,\  6,\ 17,\ 28,\ 25,\ 20,\  3,\  2,\ 21,\ 24)
\end{align*}
and
\begin{align*}
    (&25,\ 24,\  1,\  4,\ 21,\ 28,\ 29,\ 20,\ 7,\ 10,\ 17,\ 32,\ 33,\ 16,\ 13,\\
    &34+2a_1,\ 34+2a_2-1,\ 34+2a_3,\ \dots,\ 34+2a_n,\\
    &12,\ 11,\\
    &34+2a_n-1,\ 34+2a_{n-1},\ 34+2a_{n-2}-1,\ \dots,\ 34+2a_1-1,\\
    &14,\ 15,\ 34,\ 31,\ 18,\  9,\  8,\ 19,\ 30,\ 27,\ 22,\  3,\  2,\ 23,\ 26,\\
    &2n+35,\ 6,\ 5)
\end{align*}
give us irreducible meanders of orders $2n+32$ and $2n+35$, respectively. 

Thus we have the following inequalities
\begin{align*}
    \liminf_{n\to \infty} \sqrt[n]{\M^{(Irr)}_n} \geq \lim_{n\to \infty} \sqrt[n]{\M_{\frac{n-35}{2}}} = \sqrt{\lim_{n\to \infty} \sqrt[n]{\M_{n}}}.
\end{align*}

The results of~\cite{AP05} imply that $\lim\limits_{n\to \infty} \sqrt[n]{{\M}_{n}}\geq \sqrt{11.38}$, and we finally get
\begin{equation}
     \liminf_{n\to \infty} \sqrt[n]{\M^{(Irr)}_n} \geq \sqrt[4]{11.38} \approx 1.83669.
\end{equation}
\begin{remark}
This estimate could be improved if we take some suitable series of ``starting'' meanders, instead of one with the permutation $$(11, 10, 1, 2, 9, 12, 13, 8, 3, 4, 7, 14,15, 6, 5).$$
\end{remark}

\end{document}